\documentclass[10pt,letterpaper]{amsart}

\usepackage{lmodern}
\usepackage{amsmath,amssymb,amsthm,mathtools}
\usepackage[hidelinks]{hyperref}
\usepackage{microtype}

\newtheorem{theorem}{Theorem}[section]
\newtheorem{lemma}[theorem]{Lemma}
\newtheorem{proposition}[theorem]{Proposition}
\newtheorem{corollary}[theorem]{Corollary}

\newcommand{\N}{\mathcal{N}}
\newcommand{\R}{\mathbb{R}}
\newcommand{\C}{\mathbb{C}}
\newcommand{\HH}{\mathbb{H}}
\newcommand{\Ant}{\operatorname{Ant}}

\newcommand{\dd}{\,\mathrm{d}}
\newcommand{\preceqint}{\preceq}

\title{A common interleaver for two antichain polynomials on \([k]\times P_{n,s} \)}
\author[Jian Ding]{Jian Ding}
\address{College of Science, Jiujiang University, Jiujiang 332005, China}
\email{dingjianld@163.com}
\author[Lingen Ding]{Lingen Ding\textsuperscript{1}}
\address{School of Mathematics, Foshan University, Foshan 528000, China}
\email{lingending@fosu.edu.cn}

\date{}


\begin{document}
\maketitle
\footnotetext[1]{Supported by NSFC (No.~12501033) and Guangdong Basic and Applied Basic Research Foundation (No.~2024A1515110043).}

\begin{abstract}
The first author and Dong \cite{DD} proposed three conjectures on antichain generating polynomials. Jiang \cite{Jiang} recently proved Conjectures 4.3
and 4.5, concerning real-rootedness and \(\gamma\)-positivity. We prove Conjecture 4.2 by adapting his method from \([k]\times [2] \times [n]\) to \([k]\times P_{n,s} \), where \(P_{n,s}\) is a two-row Ferrers shape. We establish real stability for a family of bivariate polynomials associated with adjacent shapes, then apply the Chudnovsky-Seymour compatibility criterion to obtain a common interleaver. As noted in \cite{DD}, Conjecture 4.2 also implies Conjecture 4.3.
\end{abstract}

\section{Introduction}
Let $P$ be a finite poset.  An \emph{antichain} of $P$ is a subset whose elements are pairwise incomparable, and its antichain generating polynomial is
\[
  \N_P(x)=\sum_{A\in\Ant(P)}x^{|A|}.
\]
where \(\Ant(P)\) denotes the set of antichains of \(P\). For each ideal \(I\) of a finite poset \(Q\), inspired by the study of minuscule representations and Panyushev conjectures \cite{DW}, Ding and Dong \cite{DD} introduced a transfer polynomial $\N_I^k(x)$. These polynomials encode antichains in $[k]\times Q$ by increasing chains of ideals. Summing over all terminal ideals I gives $\N_{[k]\times Q}(x)$. Throughout we take $k\geq1$, since $\N_I^k(x)$ is not defined for $k=0$.

For $Q=[2]\times[n]$, Ding and Dong observed interlacing relations among the polynomials $\N_I^k(x)$.  Their computations led to three conjectures in Section~4 of~\cite{DD}. Conjecture~4.2 predicts common interleavers for two natural boundary sums of the transfer polynomials, Conjecture~4.3 asserts that $\N_{[2]\times[n]\times[k]}(x)$ is real-rooted for all positive $n,k$, and Conjecture~4.5 asserts that $\N_{[2]\times[n]\times[n+1]}(x)$ is palindromic and real-rooted, hence $\gamma$-positive. Conjecture 4.2 was proposed as a route to Conjecture 4.3: setting the boundary parameter to $n-1$ and applying the Chudnovsky-Seymour compatibility criterion proves real-rootedness of the full antichain polynomial.

Recently, Jiang \cite{Jiang} proved Conjectures 4.3 and 4.5 through a stronger multivariate result. He showed that the layer-refined antichain polynomial of $[2]\times[n]\times[k]$ is real stable and that its diagonal specialization has simple, strictly negative zeros. He also proved strict positivity of the $\gamma$-coefficients for $[2]\times[n]\times[n+1]$. However. Conjecture 4.2 remained open, the purpose of this paper is to prove Conjecture 4.2.

We use the following convention for interlacing.  For nonzero real-rooted polynomials $p$ and $q$, we list the zeros in weakly decreasing order. We write $p\preceqint q$ if
\[
 \deg p\leq \deg q\leq \deg p+1
\]
and the two root lists alternate as
\[
 q_1\geq p_1\geq q_2\geq p_2\geq\cdots,
\]
including every root and allowing equality.  A nonzero real-rooted polynomial $h$ is a \emph{common interleaver} of $p$ and $q$ if $p\preceqint h$ and $q\preceqint h$.

Write the ideals of $Q=[2]\times[n]$ as
\begin{equation}\label{eq:Iab}
 I(a,b)=\bigl(\{1\}\times[b]\bigr)\cup\bigl(\{2\}\times[a]\bigr),
 \qquad 0\leq a\leq b\leq n.
\end{equation}
For these ideals, write $\N^k_{(a,b)}(x)=\N^k_{I(a,b)}(x)$. Our main result is the following.

\begin{theorem}\label{thm:main}
Let $n,k\geq1$ and $0\leq s\leq n-1$.  Then the two polynomials
\begin{equation}\label{eq:L-R}
 \mathcal{L}^k_{n,s}(x)=\sum_{a=0}^{s}\sum_{b=a}^{n}\N^k_{(a,b)}(x),
 \qquad
 \mathcal{R}^k_{n,s}(x)=\sum_{b=s+1}^{n}\N^k_{(s+1,b)}(x)
\end{equation}
have a common interleaver.
\end{theorem}

Our proof extends Jiang's method \cite{Jiang} to the products \([k]\times P_{n,s} \) and the differences associated with adjacent shapes. We first obtain an enumeration formula by reflection. We then use Jacobi polynomials and positive residues to prove real stability of the resulting family. The Chudnovsky-Seymour compatibility criterion gives the required common interleaver. Taking \(s=n-1\) recovers the implication from Conjecture 4.2 to Conjecture 4.3 noted in \cite{DD}.

Section \ref{sec:enumeration} relates the boundary sums to Ferrers shapes and derives the enumeration formula. Section \ref{sec:proof} proves real stability and completes the proof of Theorem \ref{thm:main}.

\section{Ferrers shapes and antichain enumeration}\label{sec:enumeration}

For a nonnegative integer $m$, let $[m]=\{1,\dots,m\}$, with $[0]=\varnothing$.  Products of chains are ordered coordinatewise.  An \emph{ideal} $I$ of a poset $P$ is a subset such that $x\leq y$ and $y\in I$ imply $x\in I$.  We write $J(P)$ for the set of ideals of $P$ and $\max(I)$ for the set of maximal elements of $I$.

Let $Q$ be a finite poset and $k\geq1$. For each ideal $I$ of $Q$, the transfer polynomial introduced in \cite{DD} is
\begin{equation}\label{eq:transfer}
 \N_I^k(x)=
 \sum_{I_1\subseteq\cdots\subseteq I_{k-1}\subseteq I_k=I}
 x^{\sum_{j=1}^{k}|\max(I_j)\setminus I_{j-1}|},
 \qquad I_0=\varnothing,
\end{equation}
where all $I_j$ are ideals of $Q$.  Then
\begin{equation}\label{eq:sum-transfer}
  \N_{[k]\times Q}(x)=\sum_{I\in J(Q)}\N_I^k(x).
\end{equation}
This is the transfer decomposition used throughout~\cite{DD}.

We now specialize to $Q=[2]\times[n]$.  For $0\leq s\leq n$, define the two-row Ferrers ideal
\begin{equation}\label{eq:Pns}
 P_{n,s}=\bigl(\{1\}\times[n]\bigr)\cup\bigl(\{2\}\times[s]\bigr)\subseteq Q.
\end{equation}
We distinguish the two rows in the antichain polynomial of $[k]\times P_{n,s}$ by setting
\begin{equation}\label{eq:F-def}
 \mathcal{F}^k_{n,s}(u,v)=
 \sum_{A\in\Ant([k]\times P_{n,s})}
 u^{|A\cap([k]\times(\{1\}\times[n]))|}
 v^{|A\cap([k]\times(\{2\}\times[s]))|},
\end{equation}
and set
\begin{equation}\label{eq:f-def}
 f^k_{n,s}(x)=\mathcal{F}^k_{n,s}(x,x).
\end{equation}
Jiang \cite{Jiang} treated the rectangular case \(s=n\). We consider all \(0\leq s\le n\) and begin by expressing the boundary sums in terms of adjacent Ferrers shapes.

\begin{proposition}\label{prop:endpoint}
For $n,k\geq1$ and $0\leq s\leq n-1$,
\begin{equation}\label{eq:endpoint}
 \mathcal{L}^k_{n,s}(x)=f^k_{n,s}(x),
 \qquad
 \mathcal{R}^k_{n,s}(x)=f^k_{n,s+1}(x)-f^k_{n,s}(x).
\end{equation}
\end{proposition}

\begin{proof}
We use the standard bijection between increasing chains of ideals of $Q$ and
ideals of $[k]\times Q$. For a chain
\[
 I_1\subseteq\cdots\subseteq I_k,
\]
define
\begin{equation}\label{eq:J-chain}
 J=\{(p,q)\in[k]\times Q:q\in I_{k+1-p}\}.
\end{equation}
If $(p',q')\leq(p,q)\in J$, then $p'\leq p$ and $q'\leq q$. Set $j=k+1-p$ and $j'=k+1-p'$, then $j'\geq j$.  Since $I_j$ is an ideal and $I_j\subseteq I_{j'}$, it follows that $q'\in I_{j'}$, so $(p',q')\in J$. Thus $J$ is an ideal. Conversely, from an ideal $J\subseteq[k]\times Q$ define
\[
 I_j=\{q\in Q:(k+1-j,q)\in J\}.
\]
Then each $I_j$ is an ideal and $I_j\subseteq I_{j+1}$. These constructions are mutually inverse.

Under this correspondence,
\begin{equation}\label{eq:maxJ}
 \max(J)=
 \{(k+1-j,q):1\leq j\leq k,\ q\in\max(I_j)\setminus I_{j-1}\}.
\end{equation}
Indeed, if $q\in\max(I_j)\setminus I_{j-1}$ and $(p',q')\in J$ lies above $(k+1-j,q)$, then $j'=k+1-p'\leq j$ and $q'\in I_{j'}\subseteq I_j$.  Maximality forces $q'=q$, and $p'>k+1-j$ would imply $q\in I_{j-1}$, a contradiction.  The converse follows similarly: if $(p,q)$ is maximal in $J$ and $j=k+1-p$, then $q$ must be maximal in $I_j$ and cannot belong to $I_{j-1}$.  Consequently,
\begin{equation}\label{eq:max-card}
 |\max(J)|=\sum_{j=1}^{k}|\max(I_j)\setminus I_{j-1}|.
\end{equation}

Now sum~\eqref{eq:transfer} over all endpoints $I_k=I(a,b)$ with $a\leq s$.  Since the chain is increasing, this endpoint condition is equivalent to $I_j\subseteq P_{n,s}$ for every $j$, hence by~\eqref{eq:J-chain} to $J\subseteq[k]\times P_{n,s}$. The bijection $J\mapsto\max(J)$ between ideals and antichains preserves the weight by (11). Hence $\mathcal{L}^k_{n,s}=f^k_{n,s}$.

For the second identity put
\[
 E_s=\{(p,(2,s+1)):p\in[k]\}.
\]
The antichains of $[k]\times P_{n,s}$ are exactly those antichains of $[k]\times P_{n,s+1}$ that avoid $E_s$.  Under~\eqref{eq:J-chain}, an ideal $J\subseteq[k]\times P_{n,s+1}$ meets $E_s$ if and only if its terminal ideal $I_k$ contains $(2,s+1)$, i.e. if and only if $I_k=I(s+1,b)$ for some $b\geq s+1$. Since $E_s$ is an upper set of $[k]\times P_{n,s+1}$, we also have
\[
 J\cap E_s\neq\varnothing
 \quad\Longleftrightarrow\quad
 \max(J)\cap E_s\neq\varnothing.
\]
Thus $f^k_{n,s+1}-f^k_{n,s}$ is the generating polynomial for chains ending at $I(s+1, b)$ with $b \geq s + 1$, and equals $\mathcal{R}^k_{n,s}$.
\end{proof}

To compute \eqref{eq:F-def}, we encode the elements of an antichain in the two rows by four increasing sequences. We use the convention
\[
 \binom a b=0\qquad\text{if }b<0\text{ or }b>a,
\]
for all nonnegative upper arguments occurring below. 

For a finite set $E$, $\binom E r$ denotes the set of its $r$-element subsets. The next lemma is the ordinary-weight specialization of the mixed-turn reflection of Krattenthaler and Sulanke~\cite{KS}. We give the map and its inverse explicitly, keeping track of the four endpoint shifts needed for unequal row lengths.

We assume $b, c\geq 1$ in the lemma so that the upper arguments $b-1$ and $c-1$
in (13) are nonnegative. The following lemma recovers the corresponding result in \cite{Jiang}.

\begin{lemma}\label{lem:reflection}
Let $a, d, i, j$ be nonnegative integers and let $b, c$ be positive integers
such that
\[
 a\leq c,\qquad d\leq b,
 \qquad 0\leq i\leq\min(a,b),
 \qquad 0\leq j\leq\min(c,d).
\]
Consider quadruples of strictly increasing sequences
\[
 X=(x_1<\cdots<x_i)\subseteq[a],\quad
 Y=(y_1<\cdots<y_i)\subseteq[b],
\]
\[
 U=(u_1<\cdots<u_j)\subseteq[c],\quad
 V=(v_1<\cdots<v_j)\subseteq[d].
\]
The number of such quadruples for which there are no $h\in[i]$ and $\ell\in[j]$ satisfying
\begin{equation}\label{eq:crossing}
 x_h\geq u_\ell,\qquad y_h\leq v_\ell
\end{equation}
is
\begin{equation}\label{eq:reflection-count}
 \binom ai\binom bi\binom cj\binom dj
 -
 \binom{a+1}{i+1}\binom{b-1}{i-1}
 \binom{c-1}{j-1}\binom{d+1}{j+1}.
\end{equation}
\end{lemma}

\begin{proof}
Let $\mathcal A$ be the set of all quadruples, and let $\mathcal I\subseteq\mathcal A$ consist of those satisfying \eqref{eq:crossing} for at least one pair $(h, \ell)$. Then
\[
 |\mathcal A|=\binom ai\binom bi\binom cj\binom dj.
\]
If $i=0$ or $j=0$, then $\mathcal I=\varnothing$ and the correction term in~\eqref{eq:reflection-count} vanishes.  Assume henceforth that $i,j\geq1$.

Let $\mathcal C$ be the set of quadruples $(X',Y',U',V')$ with
\[
 X'\in\binom{[a+1]}{i+1},\qquad
 Y'\in\binom{[b-1]}{i-1},\qquad
 U'\in\binom{[c-1]}{j-1},\qquad
 V'\in\binom{[d+1]}{j+1}.
\]
We construct mutually inverse maps $\Gamma:\mathcal I\to\mathcal C$ and $\Delta:\mathcal C\to\mathcal I$.

Take $(X,Y,U,V)\in\mathcal I$ and set
\[
 I=\min\{h:\text{\eqref{eq:crossing} holds for some }\ell\},\qquad
 J=\min\{\ell:\text{\eqref{eq:crossing} holds for some }h\}.
\]
Then $(I,J)$ itself satisfies~\eqref{eq:crossing}.  Indeed, if $(I,\ell_0)$ and $(h_0,J)$ satisfy \eqref{eq:crossing}, then $J\leq\ell_0$ and $I\leq h_0$, so
\[
 u_J\leq u_{\ell_0}\leq x_I,
 \qquad
 y_I\leq y_{h_0}\leq v_J.
\]
Put $\alpha=u_J$ and $\beta=y_I$. The choices of $I$ and $J$ imply
\[
 h<I\Longrightarrow x_h<\alpha,
 \qquad
 \ell<J\Longrightarrow v_\ell<\beta.
\]
For $1\leq h<I$ and $1\leq\ell<J$ define
\[
 \widehat x_h=\alpha-x_{I-h},\quad
 \widehat y_h=\beta-y_{I-h},\quad
 \widehat u_\ell=\alpha-u_{J-\ell},\quad
 \widehat v_\ell=\beta-v_{J-\ell}.
\]
These numbers are positive.  Define
\begin{align*}
 X'&=(\widehat x_1,\dots,\widehat x_{I-1},\alpha,
       x_I+1,\dots,x_i+1),\\
 Y'&=(\widehat y_1,\dots,\widehat y_{I-1},
       y_{I+1}-1,\dots,y_i-1),\\
 U'&=(\widehat u_1,\dots,\widehat u_{J-1},
       u_{J+1}-1,\dots,u_j-1),\\
 V'&=(\widehat v_1,\dots,\widehat v_{J-1},\beta,
       v_J+1,\dots,v_j+1).
\end{align*}
Empty blocks are omitted.  Reversal followed by subtraction from a constant preserves strict increase, and at the splices one has
\[
 \widehat x_{I-1}<\alpha<x_I+1,
\quad
 \widehat y_{I-1}\leq\beta-1<y_{I+1}-1,
\]
\[
 \widehat u_{J-1}\leq\alpha-1<u_{J+1}-1,
\quad
 \widehat v_{J-1}<\beta<v_J+1,
\]
whenever the displayed entries exist. For each nonempty output sequence, the
corresponding endpoint bound is
\[
 \max X'\leq a+1,
 \qquad \max Y'\leq b-1,
 \qquad \max U'\leq c-1,
 \qquad \max V'\leq d+1.
\]
Thus $(X',Y',U',V')\in\mathcal C$, so $\Gamma$ is well defined.

Conversely, take $(X',Y',U',V')\in\mathcal C$. Append the auxiliary terminal entries
\[
 y'_i=b,
 \qquad
 u'_j=c,
\]
and search only among pairs $(h,\ell)\in[i]\times[j]$.  Let $I'$ and $J'$ be the least first and second coordinates occurring among pairs satisfying
\begin{equation}\label{eq:inverse-crossing}
 x'_h\leq u'_\ell,
 \qquad
 y'_h\geq v'_\ell.
\end{equation}
Such a pair exists because $x'_i\leq a\leq c=u'_j$ and $y'_i=b\geq d\geq v'_j$.  As before, $(I',J')$ itself satisfies~\eqref{eq:inverse-crossing}.  Put
\[
 \alpha'=x'_{I'},\qquad \beta'=v'_{J'}.
\]
Minimality gives
\[
 h<I'\Longrightarrow y'_h<\beta',
 \qquad
 \ell<J'\Longrightarrow u'_\ell<\alpha'.
\]
For the relevant ranges set
\[
 \widetilde x_h=\alpha'-x'_{I'-h},\quad
 \widetilde y_h=\beta'-y'_{I'-h},
\]
\[
 \widetilde u_\ell=\alpha'-u'_{J'-\ell},\quad
 \widetilde v_\ell=\beta'-v'_{J'-\ell}.
\]
Now define
\begin{align*}
 X&=(\widetilde x_1,\dots,\widetilde x_{I'-1},
       x'_{I'+1}-1,\dots,x'_{i+1}-1),\\
 Y&=(\widetilde y_1,\dots,\widetilde y_{I'-1},\beta',
       y'_{I'}+1,\dots,y'_{i-1}+1),\\
 U&=(\widetilde u_1,\dots,\widetilde u_{J'-1},\alpha',
       u'_{J'}+1,\dots,u'_{j-1}+1),\\
 V&=(\widetilde v_1,\dots,\widetilde v_{J'-1},
       v'_{J'+1}-1,\dots,v'_{j+1}-1).
\end{align*}
The auxiliary entries are used only to locate $(I',J')$ and do not appear in the output.
The same inequalities at the junctions show that the resulting sequences are strictly
increasing and lie in $[a],[b],[c],[d]$, respectively. Moreover,
\[
 x_{I'}=x'_{I'+1}-1\geq x'_{I'}=\alpha'=u_{J'},
\qquad
 y_{I'}=\beta'\leq v'_{J'+1}-1=v_{J'},
\]
so the output belongs to $\mathcal I$.

We check that the maps are mutually inverse. For an element of $\mathcal I$, the pair $(I,J)$
satisfies \eqref{eq:inverse-crossing} after applying $\Gamma$. No pair satisfying \eqref{eq:inverse-crossing}
can have $h<I$ or $\ell<J$. If both inequalities hold, reflecting back gives a pair satisfying \eqref{eq:crossing} with smaller
indices, contrary to the choice of $I$ and $J$. If only one holds, one of the inequalities
in \eqref{eq:inverse-crossing} fails by minimality. The inverse map therefore recovers $I,J,\alpha,\beta$. Reflecting
each prefix a second time restores it, the tail shifts are undone, and the crossing
coordinates return to their original rows. The reverse composition is checked in the
same way, with the inequalities reversed. Empty prefixes and tails are omitted; the
entries $x'_{i+1}$ and $v'_{j+1}$ remain available when $I'=i$ and $J'=j$. Thus the formulas
also cover the endpoint cases.

Therefore $|\mathcal I|=|\mathcal C|$, i.e.
\[
 |\mathcal I|=
 \binom{a+1}{i+1}\binom{b-1}{i-1}
 \binom{c-1}{j-1}\binom{d+1}{j+1}.
\]
Subtracting from $|\mathcal A|$ proves~\eqref{eq:reflection-count}.
\end{proof}

\begin{proposition}\label{prop:bivar}
For $n\geq1$, $0\leq s\leq n$, and $k\geq1$,
\begin{align}
 \mathcal{F}^k_{n,s}(u,v)
 =\sum_{r,t\geq0}\Bigg[&
 \binom nr\binom kr\binom st\binom kt \notag\\
 &-\binom{n-1}{r-1}\binom{k+1}{r+1}
   \binom{s+1}{t+1}\binom{k-1}{t-1}
 \Bigg]u^rv^t.
 \label{eq:bivar}
\end{align}
\end{proposition}

\begin{proof}
Suppose an antichain has $r$ elements in the long row $[k]\times(\{1\}\times[n])$ and $t$ elements in the short row $[k]\times(\{2\}\times[s])$.  Its long-row part can be written uniquely as
\[
 (b_{r+1-i},(1,a_i)),\qquad 1\leq i\leq r,
\]
where
\[
 1\leq a_1<\cdots<a_r\leq n,
 \qquad
 1\leq b_1<\cdots<b_r\leq k.
\]
Similarly the short-row part is
\[
 (d_{t+1-j},(2,c_j)),\qquad 1\leq j\leq t,
\]
with $C=\{c_1<\cdots<c_t\}\subseteq[s]$ and $D=\{d_1<\cdots<d_t\}\subseteq[k]$.  Set
\[
 \pi_i=k+1-b_{r+1-i},
 \qquad
 \rho_j=k+1-d_{t+1-j}.
\]
Then both $(\pi_i)$ and $(\rho_j)$ are strictly increasing in $[k]$. Comparability between the
rows is possible only when an element of the long row lies below an element of the
short row. This occurs precisely when
\[
 \pi_i\geq\rho_j,
 \qquad
 a_i\leq c_j.
\]
Thus the union is an antichain exactly when there is no such pair.  Apply Lemma~\ref{lem:reflection} with
\[
 (a,b,c,d,i,j)=(k,n,k,s,r,t),
 \qquad
 (X,Y,U,V)=(\pi,A,\rho,C).
\]
Here $a = c = k$, $d = s \leq n = b$, and $b, c \geq 1$, so the hypotheses of Lemma \ref{lem:reflection} hold
whenever $0 \leq r \leq min(n, k)$ and $0 \leq t \leq min(s, k)$. This gives the coefficient of $u^rv^t$ in \eqref{eq:bivar}, 
including $r = 0$ or $t = 0$. Outside these ranges no antichain with the
prescribed row sizes exists, and both terms in the coefficient formula vanish by the
binomial convention.
\end{proof}

We rewrite the coefficient formula as a product minus a rank-one correction. For this purpose, define
\begin{equation}\label{eq:HJT}
 H_{a,k}(x)=\sum_{q\geq0}\binom aq\binom kq x^q,
 \qquad
 (Jp)(x)=\frac1x\int_0^x p(w)\dd w,
 \qquad
 Tp=p-Jp.
\end{equation}
The singularity in $Jp$ at $x=0$ is removable.  If $p(x)=\sum_{q\geq0}p_qx^q$, then
\begin{equation}\label{eq:J-coeff}
 Jp(x)=\sum_{q\geq0}\frac{p_q}{q+1}x^q,
 \qquad
 Tp(x)=\sum_{q\geq0}\frac{q}{q+1}p_qx^q.
\end{equation}

\begin{corollary}\label{cor:rankone}
For $n\geq1$, $0\leq s\leq n$, and $k\geq1$,
\begin{equation}\label{eq:rankone}
 \mathcal{F}^k_{n,s}(u,v)
 =H_{n,k}(u)H_{s,k}(v)
 -\frac{(k+1)(s+1)}{nk}
 (TH_{n,k})(u)(TH_{s,k})(v).
\end{equation}
\end{corollary}

\begin{proof}
For every $r,t\geq0$,
\[
 \binom{n-1}{r-1}\binom{k+1}{r+1}
 =\frac{k+1}{n}\frac{r}{r+1}\binom nr\binom kr,
\]
and
\[
 \binom{s+1}{t+1}\binom{k-1}{t-1}
 =\frac{s+1}{k}\frac{t}{t+1}\binom st\binom kt.
\]
Substituting these identities into \eqref{eq:bivar} and using \eqref{eq:J-coeff} gives \eqref{eq:rankone}.
\end{proof}

\section{Proof of the main theorem}\label{sec:proof}

We now use Corollary \ref{cor:rankone} to prove real stability of the nonnegative boundary pencils. We begin with the definitions. Let
\[
 \HH=\{z\in\C:\operatorname{Im}z>0\}.
\]
A polynomial with real coefficients in several variables is \emph{real stable} if it is nonzero whenever all variables lie in $\HH$.  In one variable, a nonzero real polynomial is real stable if and only if all of its zeros are real. In this paper, a \emph{Pick function} is a holomorphic map from $\HH$ into $\HH$. 

The relation between $H_{s,k}$ and $H_{s+1,k}$ that we need is
\begin{equation}\label{eq:adjacent}
 \frac{s+2}{k}TH_{s+1,k}-\frac{s+1}{k}TH_{s,k}
 =\frac{s+1}{k}\bigl(H_{s+1,k}-H_{s,k}\bigr).
\end{equation}
Both sides have zero constant term. For $q\geq1$, the coefficient of $x^{q}$ on the right is
\[
 \frac{s+1}{k}\binom{s}{q-1}\binom{k}{q},
\]
and the coefficient on the left has the same value by
\[
 \binom sq=\frac{s-q+1}{q}\binom{s}{q-1}.
\]

For $a,k\geq1$, put
\begin{equation}\label{eq:Q-def}
 Q_{a,k}(x)=JH_{a,k}(x)
 ={}_2F_1(-a,-k;2;x).
\end{equation}

\begin{lemma}\label{lem:jacobi}
Let $a,k\geq1$, $d=\min(a,k)$, and $\beta=|a-k|$.  Then
\begin{equation}\label{eq:jacobi}
 Q_{a,k}(x)=
 \frac{(1-x)^d}{d+1}
 P_d^{(1,\beta)}\!\left(\frac{1+x}{1-x}\right).
\end{equation}
Consequently, all $d$ zeros of $Q_{a,k}$ are simple and lie in $(-\infty,0)$.
\end{lemma}

\begin{proof}
The hypergeometric definition of the Jacobi polynomial is
\[
 P_d^{(1,\beta)}(y)
 =(d+1)\,{}_2F_1\!\left(-d,d+\beta+2;2;\frac{1-y}{2}\right).
\]
Set $y=(1+x)/(1-x)$, so $(1-y)/2=-x/(1-x)$.  Pfaff's transformation
\[
 {}_2F_1(A,B;C;z)
 =(1-z)^{-A}{}_2F_1\!\left(A,C-B;C;\frac{z}{z-1}\right)
\]
with $A=-d$, $B=d+\beta+2$, $C=2$, and $z=-x/(1-x)$ yields \eqref{eq:jacobi}.  Since $\{d,d+\beta\}=\{a,k\}$, the resulting hypergeometric polynomial is $Q_{a,k}$.  Both sides of \eqref{eq:jacobi} are polynomials in $x$, so the identity holds at $x=1$ as well.

Since $1,\beta > -1$, the Jacobi polynomial has $d$ simple zeros in $(-1, 1)$; see \cite[Chapter~III, \S3.3]{Szego}. The inverse M\"obius transformation $x=(y-1)/(y+1)$ maps these to d simple zeros of $Q_{a,k}$ in $(-\infty,0)$.
\end{proof}

\begin{lemma}\label{lem:diff}
For $a,k\geq1$ and $Q=Q_{a,k}$,
\begin{equation}\label{eq:diff}
 H_{a,k}=Q+xQ',
 \qquad
 TH_{a,k}=xQ',
 \qquad
 H_{a+1,k}-H_{a,k}=x(kQ-xQ').
\end{equation}
\end{lemma}

\begin{proof}
The first identity follows from $Q=JH_{a,k}$ and~\eqref{eq:J-coeff}; subtracting $Q$ gives the second.  For $q\geq1$, the coefficient of $x^q$ in $x(kQ-xQ')$ is
\[
 (k-q+1)\frac1q\binom{a}{q-1}\binom{k}{q-1}
 =\binom{a}{q-1}\binom{k}{q},
\]
which, by Pascal's identity, is the coefficient of $x^q$ in $H_{a+1,k}-H_{a,k}$.  Both constant coefficients are zero.
\end{proof}

\begin{lemma}\label{lem:pick}
Fix integers $a\geq 0$ and $k\geq 1$, and write
\[
 B=H_{a,k},
 \qquad
 \Delta=H_{a+1,k}-H_{a,k}.
\]
For $\lambda\geq0$ put
\begin{equation}\label{eq:C-N}
 C_\lambda=B+\lambda\Delta,
 \qquad
 N_\lambda=TB+\lambda\Delta.
\end{equation}
Every zero of $C_\lambda$ is simple and negative.  Moreover, $N_\lambda/C_\lambda$ is a Pick function, except that it is identically zero when $(a,\lambda)=(0,0)$.
\end{lemma}

\begin{proof}
Assume first that $a\geq1$.  Let $d=\min(a,k)$ and write the zeros of $Q=Q_{a,k}$ as
\[
 \rho_1<\rho_2<\cdots<\rho_d<0.
\]
By Lemma~\ref{lem:diff},
\begin{equation}\label{eq:CQ}
 C_\lambda=Q+xQ'+\lambda x(kQ-xQ'),
 \qquad
 \\N_\lambda=C_\lambda-Q.
\end{equation}
At a zero $\rho_i$ of $Q$,
\begin{equation}\label{eq:C-rho}
 C_\lambda(\rho_i)=\rho_iQ'(\rho_i)(1-\lambda\rho_i).
\end{equation}
The last factor is positive.  Since the leading coefficient of $Q$ is positive,
\[
 \operatorname{sgn}Q'(\rho_i)=(-1)^{d-i},
\]
and hence
\begin{equation}\label{eq:C-sign}
 \operatorname{sgn}C_\lambda(\rho_i)=(-1)^{d-i+1}.
\end{equation}
Also $C_\lambda(0)=1$.  Therefore $C_\lambda$ has a zero in each interval
\begin{equation}\label{eq:intervals}
 (\rho_i,\rho_{i+1})\quad(1\leq i<d),
 \qquad
 (\rho_d,0).
\end{equation}

To determine whether there are further zeros, we compute the degree of $C_\lambda$. If $\lambda=0$, then $C_\lambda=H_{a,k}$ has degree $d$.  If $\lambda>0$ and $a\geq k$, then $d=k$ and
\[
 [x^k]C_\lambda=\binom ak+\lambda\binom a{k-1}>0,
\]
so $\deg C_\lambda=d$.  If $\lambda>0$ and $a<k$, then $d=a$ and the coefficient of $x^{d+1}$ in $\lambda x(kQ-xQ')$ is
\[
 \lambda(k-d)[x^d]Q=\lambda\binom{k}{a+1}>0,
\]
so $\deg C_\lambda=d+1$.  Thus
\begin{equation}\label{eq:degree-C}
 \deg C_\lambda=
 \begin{cases}
 d,&\lambda=0,\\
 d,&\lambda>0\text{ and }a\geq k,\\
 d+1,&\lambda>0\text{ and }a<k.
 \end{cases}
\end{equation}
When $\deg C_\lambda= d$, each of the d intervals in \eqref{eq:intervals} contains exactly one zero, which
must be simple. When $\deg C_\lambda= d+1$, the positive leading coefficient and \eqref{eq:C-sign} give
an additional zero in $(-\infty,\rho_1)$. These account for all the zeros, so each is simple
and negative.

Let the zeros of $C_\lambda$ be $c_i$.  In the degree-$d$ case they satisfy
\[
 \rho_1<c_1<\rho_2<c_2<\cdots<\rho_d<c_d<0,
\]
and in the degree-$d+1$ case
\[
 c_0<\rho_1<c_1<\cdots<\rho_d<c_d<0.
\]
Since $Q$ and $C_\lambda$ have no common zero, every residue
\begin{equation}\label{eq:residue}
 R_i=\frac{Q(c_i)}{C_\lambda'(c_i)}
\end{equation}
is nonzero.  Interlacing shows that $Q(c_i)$ and $C_\lambda'(c_i)$ have the same sign, hence
\begin{equation}\label{eq:res-positive}
 R_i>0.
\end{equation}
Partial fractions therefore give, for $z=x+iy\in\HH$,
\begin{equation}\label{eq:negative-im}
 \operatorname{Im}\frac{Q(z)}{C_\lambda(z)}
 =-y\sum_i\frac{R_i}{|z-c_i|^2}<0.
\end{equation}
Using $N_\lambda/C_\lambda=1-Q/C_\lambda$ from~\eqref{eq:CQ}, we obtain
\[
 \operatorname{Im}\frac{N_\lambda(z)}{C_\lambda(z)}>0.
\]
Thus $N_\lambda/C_\lambda$ is a Pick function.

If $a=0$, then $B=1$, $\Delta=kx$, and
\[
 C_\lambda=1+\lambda kx,
 \qquad
 N_\lambda=\lambda kx.
\]
For $\lambda=0$ the quotient is identically zero.  For $\lambda>0$, $C_\lambda$ has the single simple negative zero $-1/(\lambda k)$ and
\[
 \operatorname{Im}\frac{\lambda kz}{1+\lambda kz}
 =\frac{\lambda k\operatorname{Im}z}{|1+\lambda kz|^2}>0.
\]
\end{proof}

The Pick function property can be used to prove the following important real stability result.

\begin{theorem}\label{thm:stable-pencil}
Let $n,k\geq1$ and $0\leq s\leq n-1$.  For every $\lambda\geq0$, the bivariate polynomial
\begin{equation}\label{eq:Phi}
 \Phi_\lambda(u,v)=
 F^k_{n,s}(u,v)
 +\lambda\bigl(F^k_{n,s+1}(u,v)-F^k_{n,s}(u,v)\bigr)
\end{equation}
is real stable.
\end{theorem}

\begin{proof}
Put
\[
 A=H_{n,k},\qquad
 B=H_{s,k},\qquad
 \Delta=H_{s+1,k}-H_{s,k},
 \qquad
 \eta=\frac{(k+1)(s+1)}{nk}.
\]
Applying Corollary \ref{cor:rankone} to the two adjacent shapes and using \eqref{eq:adjacent}, we obtain
\begin{equation}\label{eq:Phi-factor}
 \Phi_\lambda(u,v)
 =A(u)C_\lambda(v)-\eta(TA)(u)N_\lambda(v),
\end{equation}
where
\[
 C_\lambda=B+\lambda\Delta,
 \qquad
 N_\lambda=TB+\lambda\Delta.
\]
Let $u,v\in\HH$.  Lemma~\ref{lem:pick}, applied to $(a,\lambda)=(n,0)$ and $(a,\lambda)=(s,\lambda)$, shows that $A(u)C_\lambda(v)\neq0$ and that
\[
 R_1(u)=\frac{(TA)(u)}{A(u)},
 \qquad
 R_2(v)=\frac{N_\lambda(v)}{C_\lambda(v)}
\]
both lie in $\HH$, except that $R_2\equiv0$ when $s=0$ and $\lambda=0$.  Dividing~\eqref{eq:Phi-factor} by $A(u)C_\lambda(v)$ gives
\begin{equation}\label{eq:Phi-normalized}
 \frac{\Phi_\lambda(u,v)}{A(u)C_\lambda(v)}
 =1-\eta R_1(u)R_2(v).
\end{equation}
The product of two points of $\HH$ cannot be a positive real number.  Indeed, if $z_r=x_r+iy_r$ with $y_r>0$ and $\operatorname{Im}(z_1z_2)=0$, then
\[
 x_2=-\frac{x_1y_2}{y_1},
 \qquad
 \operatorname{Re}(z_1z_2)
 =-\frac{y_2}{y_1}(x_1^2+y_1^2)<0.
\]
Since $1/\eta>0$, the right-hand side of~\eqref{eq:Phi-normalized} cannot vanish.  In the exceptional case $s=0$, $\lambda=0$, we have $C_\lambda=1$, $N_\lambda=0$, and $\Phi_\lambda(u,v)=A(u)\neq0$.  Hence $\Phi_\lambda$ is real stable.
\end{proof}

It remains to deduce the existence of a common interleaver from real stability. Two nonzero real polynomials $p,q$ are called \emph{compatible} if
\[
 \alpha p+\beta q
\]
is real-rooted for all $\alpha,\beta\geq0$ not both zero. The Chudnovsky--Seymour criterion~\cite[Theorem~3.6]{CS} relates compatibility to common interleavers as follows.

\begin{theorem}\label{thm:CS}
A finite family of real-rooted real polynomials with positive leading coefficients is compatible if and only if their root sequences admit a common interleaving.  In particular, a compatible pair with positive leading coefficients has a common interleaver.
\end{theorem}

We will also need the fact that real-rootedness is preserved under coefficient
limits of bounded degree, provided the limit is nonzero.

\begin{lemma}\label{lem:RT}
Fix a nonnegative integer \(d\ge 0\), and let \((p_m)\) be a sequence of real-rooted real polynomials
of degree at most \(d\). View their coefficient vectors as elements of \(\R^{d+1}\) by padding with zeros. If these vectors converge to the coefficient vector of a nonzero polynomial \(p\), then \(p\) is real-rooted.
\end{lemma}

\begin{proof}
Suppose \(p\) has a nonreal zero \(z_0\). Choose a closed disk \(\mathcal D\) centered at \(z_0\) whose closure is disjoint from the real axis and whose boundary contains no zero of \(p\). Coefficient convergence implies uniform convergence on \(\partial \mathcal D\). For all sufficiently large \(m\),
\[
|p_m(z)-p(z)|<|p(z)|\qquad(z\in\partial \mathcal D).
\]
By Rouch\'e's theorem, \(p_m\) and \(p\) have the same number of zeros in \(\mathcal D\), counted with multiplicity. This contradicts the real-rootedness of \(p_m\).
\end{proof}

We now return to Theorem \ref{thm:main}.

\begin{proof}[Proof of Theorem~\ref{thm:main}]
For $\lambda\geq0$, consider the diagonal specialization
\begin{equation}\label{eq:p-lambda}
 p_\lambda(x)=
 f^k_{n,s}(x)
 +\lambda\bigl(f^k_{n,s+1}(x)-f^k_{n,s}(x)\bigr)
\end{equation}
This polynomial is real-rooted by Theorem \ref{thm:stable-pencil}. Indeed, for $z\in\HH$ we have $p_\lambda(z)=\Phi_\lambda(z,z)\neq0$. Since the coefficients are real, $p_\lambda$ has no nonreal zeros. Its constant term is 1, so it is nonzero.

Set
\[
 f=f^k_{n,s},
 \qquad
 g=f^k_{n,s+1}-f^k_{n,s}.
\]
Then $f=p_0$ is real-rooted. By Proposition \ref{prop:endpoint}, $g$ is the generating polynomial for
antichains of $[k]\times P_{n,s+1}$ that meet
\[
 E_s=\{(p,(2,s+1)):p\in[k]\}.
\]
Exactly $k$ one-element antichains meet $E_s$, so
\begin{equation}\label{eq:xg}
 [x]g=k>0.
\end{equation}
In particular, $g\neq0$.

For every positive integer $m$,
\[
 q_m=m^{-1}p_m=m^{-1}f+g
\]
is real-rooted. The coefficients of $q_m$ converge to those of $g$, and the degrees are
bounded by $\max(\deg f,\deg g)$. Since $g \neq 0$, Lemma \ref{lem:RT} shows that $g$ is real-rooted.

If $\alpha>0$ and $\beta\geq0$, then
\[
 \alpha f+\beta g=\alpha p_{\beta/\alpha}
\]
is a nonzero real-rooted polynomial.  If $\alpha=0$, then $\beta>0$ and $\beta g$ is nonzero and real-rooted.  Hence $f$ and $g$ are compatible. Both polynomials have nonnegative coefficients. Since $[x^0]f=1$ and $[x]g=k>0$ by \eqref{eq:xg}, their leading coefficients are
positive. Theorem 3.5 therefore gives a common interleaver for $f$ and $g$. Finally, Proposition~\ref{prop:endpoint} identifies
\[
 (f,g)=\bigl(\mathcal{L}^k_{n,s},\mathcal{R}^k_{n,s}\bigr),
\]
which proves the theorem.
\end{proof}


\begin{thebibliography}{99}

\bibitem{CS}
M.~Chudnovsky and P.~Seymour,
\emph{The roots of the independence polynomial of a clawfree graph},
J. Combin. Theory Ser. B \textbf{97} (2007), 350--357.

\bibitem{DD}
J.~Ding and C.~P.~Dong,
\emph{Antichain generating polynomials of posets},
arXiv:1905.06692v1 [math.CO] (2019).

\bibitem{DW}
C.~P.~Dong and G.~Weng, \emph{Minuscule representations and Panyushev conjectures}, Sci. China Math. 61 (2018), no.10, 1759-1774

\bibitem{Jiang}
W.~Jiang,
\emph{Real stability of layer-refined antichain polynomials for three-chain products with a two-element factor},
preprint (2026).

\bibitem{KS}
C.~Krattenthaler and R.~A.~Sulanke,
\emph{Counting pairs of nonintersecting lattice paths with respect to weighted turns},
Discrete Math. \textbf{153} (1996), 189--198.

\bibitem{Szego}
G.~Szeg\H{o},
\emph{Orthogonal Polynomials}, 4th ed.,
American Mathematical Society, Providence, RI, 1975.

\end{thebibliography}
\end{document}